\documentclass[11pt,amssymb]{amsart}
\usepackage{amssymb,hyperref}
\usepackage[mathscr]{eucal}
\usepackage[all,cmtip]{xy}
\usepackage{tikz-cd}

\newcommand{\Z}{{\mathbb Z}}

\newcommand{\Gal}{\mathrm{Gal}}

\newtheorem{thm}{Theorem}[section]
\newtheorem{lemma}[thm]{Lemma}
\newtheorem{prop}[thm]{Proposition}
\newtheorem{cor}[thm]{Corollary}
\theoremstyle{definition}
\newtheorem{remark}[thm]{Remark}

\newcommand{\Hom}{\mathrm{Hom}}

.240pk scaled 1200 .240pk
\DeclareFontFamily{U}{wncy}{}
    \DeclareFontShape{U}{wncy}{m}{n}{<->wncyr10}{}
    \DeclareSymbolFont{mcy}{U}{wncy}{m}{n}
    \DeclareMathSymbol{\Sha}{\mathord}{mcy}{"58}
\begin{document}

\title[The finiteness of the Tate-Shafarevich group]{\bf Finiteness of the Tate-Shafarevich group over function fields for groups of multiplicative type}

\begin{abstract}
Let $K = k(X)$ be the function field of a smooth geometrically integral variety $X$ of dimension $\geq 2$ over a field $k$ of characteristic 0 and $V$ be the set of discrete valuations of $K$ associated with the prime divisors on $X$. We show that if $D$ is a $k$-defined group of multiplicative type, then the corresponding Tate-Shafarevich group $\Sha(D,V) = \ker \left(H^1(K,D) \to \prod_{v \in V} H^1(K_v, D) \right)$ is finite in the following situations: (1) $k$ is finitely generated and $X(k) \neq \emptyset$; (2) $k$ is a number field. This complements previous work of Harari and Szamuely \cite{HS}, which considered the case where $X$ is a curve.
\end{abstract}

\author[I.A.~Rapinchuk]{Igor A. Rapinchuk}

\author[A.~Roy]{Avinash Roy}

\address{Department of Mathematics, Michigan State University, East Lansing, MI 48824, USA}

\email{rapinchu@msu.edu}

\address{Department of Mathematics, Michigan State University, East Lansing, MI 48824, USA}

\email{royavina@msu.edu}

\maketitle

\section{Introduction}\label{S:intro}
Suppose $K$ is a field equipped with a set $V$ of rank one valuations and $G$ is an algebraic group defined over $K$. One defines the corresponding \emph{Tate-Shafarevich set} $\Sha(G,V)$ as the kernel of the global-to-local map in Galois cohomology
$$
\Sha(G,V) = \ker \left( H^1(K,G) \stackrel{\theta_{G,V}}{\longrightarrow} \prod_{v \in V} H^1(K_v, G) \right),
$$
where $K_v$ denotes the completion of $K$ with respect to $v$.

A major theme in number theory and arithmetic geometry has been the study of the finiteness properties of $\Sha(G,V)$ for various kinds of algebraic groups. In the context of linear algebraic groups, one of the central results is that if
$K$ is a number field, $V$ is the set of all places of $K$, and $G$ is either a simply-connected or adjoint semisimple $K$-group, then $\theta_{G,V}$ is injective, and hence $\Sha(G,V)$ is trivial (see \cite[Chapter 6]{Pl-R}). Furthermore, although $\theta_{G,V}$ may fail to be injective for an arbitrary $K$-group $G$, it turns that $\theta_{G,V}$ is nevertheless always \emph{proper} (i.e., the preimage of a finite set is finite), so, in particular, $\Sha(G,V)$ is \emph{finite} (see \cite[Ch. III, \S\S4.6-4.7]{Serre-GC}).

For a long time, it had been thought that this finiteness result is specific to number fields, where its proof in the general case relies on reduction theory for adelic groups. However, in recent years, there has been growing evidence that a similar finiteness statement may hold in a far more general setting (we refer the reader to \cite[\S6]{RR-Survey} for an overview of some of these developments).

A natural (higher-dimensional) situation to consider is when $K = k(X)$ is the function field of a smooth geometrically integral variety $X$ over a field $k$ and $V$ is the set discrete valuations of $K$ corresponding to the prime divisors of $X$ (``geometric places"). One can then ask about the finiteness properties of $\Sha(G,V)$ for an algebraic $K$-group $G$. Although, at present, this problem is largely open, we would like to mention two available results. First, if $k$ is a field of characteristic 0 satisfying Serre's condition (F) (see \cite[Ch. III, \S4.1]{Serre-GC}), then it was shown  in \cite{RR-Tori} and \cite{RR-Tori1} that $\Sha(D,V)$ is finite when $D$ is an algebraic $K$-group of multiplicative type. The argument was based on adelic considerations and, in fact, applies, more generally, to all (possibly non-commutative) $K$-groups $D$ whose connected component $D^{\circ}$ is a torus.

Second, suppose $k$ is a number field (hence, in particular, does not satisfy condition (F)). Then Harari and Szamuely \cite{HS} showed
that a similar finiteness statement for $\Sha(D,V)$ can be established when $K = k(C)$ is the function field of a curve $C$ over $k$ provided
one only considers groups that are \emph{defined over $k$}. More precisely, they proved that if $D$ is a group of multiplicative type defined over a field $k$ of characteristic 0 and $K= k(C)$ is the function field of a smooth geometrically integral $k$-defined curve $C$, then $\Sha(D,V)$ is finite when
\begin{enumerate}
\item $k$ is a finitely generated field and $C(k) \neq \emptyset$; or
\item $k$ is a number field.
\end{enumerate}
A bit later, it was observed in \cite{RR-Tori2} that the adelic approach of \cite{RR-Tori} and \cite{RR-Tori1} can be adapted to prove the finiteness of $\Sha(T,V)$ when $X$ is an arbitrary smooth geometrically integral variety and $T$ is a $k$-defined algebraic torus (in particular, this provides an alternative proof of the main result of \cite{HS} in the case of tori). The purpose of this note is to extend the latter finiteness statement to all groups of multiplicative type defined over $k$. Since the case of curves was considered in \cite{HS}, to avoid repetition, we will only consider varieties of dimension $\geq 2$ (although our arguments, with some simplifications, also work for curves). Thus, our
main result is the following.

\begin{thm}\label{T:mainthm}
Suppose $X$ is a smooth geometrically integral variety  of dimension $\geq 2$ over a field $k$ of characteristic 0. Let $V$ be the set of discrete valuations of the function field $K = k(X)$ corresponding to the codimension one points of $X$. Then for any $k$-defined group $D$ of multiplicative type,
the Tate-Shafarevich group $\Sha(D, V)$ is finite in the following situations:
    \begin{enumerate}
    \item $k$ is a finitely generated field and $X(k) \neq \emptyset$; or
\item $k$ is a number field.
    \end{enumerate}

\end{thm}

\begin{remark}
For our purposes, we may (and will) assume that $X$ is an affine variety. Indeed, this
can always be achieved by replacing $X$ with a suitable affine open $k$-subvariety, which will only shrink the corresponding set $V$ of geometric places of $K = k(X)$.
\end{remark}

The paper is organized as follows. In \S\ref{S:Finite-group}, we establish the properness of the global-to-local map for finite commutative $k$-defined groups. Next, in \S\ref{S:adeles}, we briefly review some key points concerning adelic groups. These considerations are then applied in \S\ref{S:Finiteness-Rel-TS} to obtain a finiteness result for a certain relative Tate-Shafarevich group. Finally, in \S\ref{S:proof-of-main-theorem}, we use the results of \S\ref{S:Finite-group} and \S\ref{S:Finiteness-Rel-TS} to prove Theorem \ref{T:mainthm}.

\vskip3mm

\noindent {\bf Notations and conventions.} Suppose $G$ is an algebraic group over a field $K$. We will use the usual notations (see \cite{Serre-GC}) for the Galois cohomology of $G$. Namely, if $L/K$ is a (finite or infinite) Galois extension, we will denote by $H^1(L/K, G)$ the (pointed) Galois cohomology set $H^1(\Gal(L/K), G(L))$; if $L = K^{\rm sep}$ is a separable closure of $K$, we will denote this set simply by $H^1(K, G)$. For additional clarity, we will sometimes write $H^1(L/K, G(L)).$

\section{The case of a finite commutative group}\label{S:Finite-group}

The goal of this section is to prove the following statement concerning the properness of the global-to-local map for finite commutative groups.

\begin{prop}\label{P:propfin}
Let $X$ be a smooth geometrically integral affine variety of dimension $\geq 2$ over a field $k$ of characteristic 0, and denote by $V$ the set of
geometric places of
the function field $K = k(X)$ corresponding to the codimension one points of $X.$ Suppose $\Omega$ is a finite commutative algebraic group defined over $k$. Then the global-to-local map
$$
H^1(K, \Omega) \stackrel{\theta_{\Omega, V}}{\longrightarrow} \prod_{v \in V} H^1(K_v, \Omega)
$$
is proper.

\end{prop}


\begin{proof}
It suffices to show that $\ker \theta_{\Omega, V} = \Sha(\Omega, V)$ is finite. First, we observe that, using the inflation-restriction sequence, we may replace $K$ by a finite extension $L/K$ (see the proof of \cite[Proposition 2.1]{RR-Tori1} for the details). So, let $\ell/k$ be a finite Galois extension such that $X(\ell) \neq \emptyset$ and $\Omega$ is a trivial Galois module over $\ell$ (more precisely, $\Omega$ is isomorphic over $\ell$ to a product of cyclic groups $\Z/m \Z$ with trivial Galois action). Then, setting $L = \ell \cdot K$, we note that $L/K$ is a finite Galois extension. Let $V^L$ be the set of all extensions of the places in $V$ to $L$ (in geometric terms, the elements of $V^L$ correspond to the prime divisors of the normalization of $X$ in $L$). Then it suffices to show the finiteness of the kernel of the map
$$
H^1(L, \Omega) \to \prod_{w \in V^L} H^1(L_w, \Omega).
$$
Thus, replacing $k$ by $\ell$, we may assume that $X(k) \neq \emptyset$ and $\Omega$ is a trivial Galois module over $K$ (and over $k$). In particular, we have
$$
H^1(K, \Omega) = {\rm Hom}_{\rm cont}(\Gal(K^{\rm sep}/K), \Omega) \ \ \ \text{and} \ \ \ H^1(K_v, \Omega) = {\rm Hom}_{\rm cont}(\Gal(K_v^{\rm sep}/K_v), \Omega).
$$

Suppose now that $t \in \ker \theta_{\Omega, V}.$ Then $t$ is represented by a continuous homomorphism
$$
\chi \colon \Gal(K^{\rm sep}/K) \to \Omega
$$
that vanishes on every decomposition group $G(v) = \Gal(K_v^{\rm sep}/K_v)$ for $v \in V$. In particular, $\chi$ vanishes on all of the inertia subgroups and thus factors through a homomorphism $\chi' \colon \Gal(K_V/K) \to \Omega$, where $K_V/K$ is the maximal subextension of $K^{\rm sep}$ that is unramified at all $v \in V$.  Thus, $t$ is in the image of the inflation map
$$
H^1(\Gal(K_V/K), \Omega) \to H^1(K, \Omega).
$$
Since the inflation map on degree 1 cohomology is injective, we may identify $H^1(\Gal(K_V/K), \Omega)$ with its image in $H^1(K, \Omega)$, and thus, by a slight abuse of notation, view $\Sha(\Omega, V)$ as a subset of $H^1(\Gal(K_V/K), \Omega) = {\rm Hom}_{\rm cont}(\Gal(K_V/K), \Omega).$

Now, it is known that $\Gal(K_V/K)$ can be canonically identified with the \'etale fundamental group $\pi_1^{\mathrm{\acute{e}t}}(X, \bar{s})$ for the geometric point $\bar{s} \colon {\rm Spec}~(K^{\rm sep}) \to X$ (see the argument in \cite[\S2]{RR-Tori} for the details). Furthermore, we have the homotopy sequence
$$
1 \to \pi_1^{\mathrm{\acute{e}t}}(\overline{X}, \bar{s}) \to \pi_1^{\mathrm{\acute{e}t}}(X, \bar{s}) \to \Gal(k^{\rm sep}/k) \to 1
$$
where $\overline{X} = X \times_k k^{\rm sep}$ (see, e.g., \cite[Proposition 5.6.1]{Szamuely}), which leads to the exact sequence of abelian groups
$$
0 \to \Hom_{\rm cont}(\Gal(k^{\rm sep}/k), \Omega) \to \Hom_{\rm cont}(\pi_1^{\mathrm{\acute{e}t}}(X, \bar{s}), \Omega) \to \Hom_{\rm cont}(\pi_1^{\mathrm{\acute{e}t}}(\overline{X}, \bar{s}), \Omega).
$$
As above, we view $\Sha(\Omega, V)$ as a subset of $\Hom_{\rm cont}(\pi_1^{\mathrm{\acute{e}t}}(X, \bar{s}), \Omega)$. Since $k$ is a field of characteristic 0, it is known that $\pi_1^{\mathrm{\acute{e}t}}(\overline{X}, \bar{s})$ is a finitely presented profinite group (see the discussion in \cite[\S1]{Esn}). Consequently, $\Hom_{\rm cont}(\pi_1^{\mathrm{\acute{e}t}}(\overline{X}, \bar{s}), \Omega)$ is finite.

To complete the argument, we will show that
the image of the group $\Hom_{\rm cont}(\Gal(k^{\rm sep}/k), \Omega)$ in $\Hom_{\rm cont}(\pi_1^{\mathrm{\acute{e}t}}(X, \bar{s}), \Omega)$ has trivial intersection with $\Sha(\Omega, V).$ For this, we will need the following Bertini-type result that was formulated in
\cite[Proposition 10]{RR-Tori2} (and which immediately follows from \cite[Theorem 3.6]{GK}).

\begin{prop}\label{P:Bertini-type}
Let $X$ be a smooth geometrically integral quasi-projective variety of dimension $\geq 2$ over a field $k$ of characteristic 0. Then $X$ contains a smooth geometrically integral $k$-defined subvariety $Y$ of codimension 1. Moreover, given $x \in X(k)$, one can choose such a $Y$ so that it contains $x$.
\end{prop}

Take $y \in X(k)$. Using Proposition \ref{P:Bertini-type}, we can find a geometrically integral smooth $k$-defined subvariety $Y \subset X$ of codimension 1 such that $y \in Y(k)$. Let $v = v_Y$ be the discrete valuation of $K$ associated with $Y$. Suppose $z \in H^1(k, \Omega) = {\rm Hom}_{{\rm cont}}(\Gal(k^{\rm sep}/k), \Omega)$ is an element whose image in $\Hom_{\rm cont}(\pi_1^{\mathrm{\acute{e}t}}(X, \bar{s}), \Omega)$ lies in $\Sha(\Omega, V).$ We represent $z$ by an element
$$
\tilde{z} \in H^1(\tilde{\ell}/k, \Omega)
$$
for a finite Galois extension $\tilde{\ell}/k$.
Set $\tilde{L} = \tilde{\ell} \cdot K$ and let $\tilde{v}$ be a place of $\tilde{L}$ above $v$. Passing to the completions, we note that $\tilde{L}_{\tilde{v}}/K_{v}$ is an unramified extension with corresponding residue fields $\tilde{L}^{(\tilde{v})} = \tilde{\ell}(Y_{\tilde{\ell}})$ and $K^{(v)} = k(Y)$ (where $Y_{\tilde{\ell}} = Y \times_k \tilde{\ell}$). Consequently, reduction modulo the maximal ideal of the valuation ring $\mathcal{O}_{\tilde{L}_{\tilde{v}}}$ induces a map
$$
\rho_v \colon H^1(\tilde{L}_{\tilde{v}}/K_v, \Omega) \to H^1(\tilde{L}^{(\tilde{v})}/K^{(v)}, \Omega) = H^1(\tilde{\ell}(Y_{\tilde{\ell}})/k(Y), \Omega).
$$
Next, consider the local rings $\mathcal{O}_{Y, y}$ and $\mathcal{O}_{Y_{\tilde{\ell}}, y} = \tilde{\ell} \cdot \mathcal{O}_{Y, y}$ with maximal ideals $\mathfrak{m}$ and $\mathfrak{m}_{\tilde{\ell}}.$ The residue fields $\mathcal{O}_{Y,y}/\mathfrak{m}$ and $\mathcal{O}_{Y_{\tilde{\ell}}, y}/\mathfrak{m}_{\tilde{\ell}}$ coincide with $k$ and $\tilde{\ell}$, respectively, and we again have a map
$$
\rho_{Y_{\tilde{\ell}}, y} \colon H^1(\tilde{\ell}(Y_{\tilde{\ell}})/k(Y), \Omega) \to H^1(\tilde{\ell}/k, \Omega)
$$
induced by reduction modulo $\mathfrak{m}_{\tilde{\ell}}.$ These fit together into a diagram
\[\begin{tikzcd}
	{H^1(\tilde{\ell}/k, \Omega)} & {H^1(\tilde{L}_{\tilde{v}}/K_v, \Omega)} \\
	& {H^1(\tilde{\ell}(Y_{\tilde{\ell}})/k(Y), \Omega)} \\
	& {H^1(\tilde{\ell}/k, \Omega)}
	\arrow["{\theta_v}", from=1-1, to=1-2]
	\arrow["{{\rm id}}"', dashed, from=1-1, to=3-2]
	\arrow["{\rho_v}", from=1-2, to=2-2]
	\arrow["{\rho_{Y_{\tilde{\ell}}, y}}", from=2-2, to=3-2]
\end{tikzcd}\]
such that the composition $\rho_{Y_{\tilde{\ell}}, y} \circ \rho_v \circ \theta_v$ is the identity map. Since $\tilde{z} \in \ker \theta_v$, it follows that $\tilde{z} = 0$, as needed.
\end{proof}

\section{Brief review of adelic groups}\label{S:adeles}

In this section, we quickly recall the essential points concerning adelic groups that will be needed later in the paper.
For further details, the reader can consult \cite[\S3]{RR-Tori} and \cite[\S3]{RR-Tori1}. Although in subsequent sections we will only work with groups of multiplicative type, in this section, we will consider, more generally, possibly non-commutative algebraic groups whose connected component is a torus.

Suppose $K$ is a field equipped with a set $V$ of discrete valuations satisfying the following condition (which holds automatically for the sets of geometric places of function fields):
\vskip2mm
\noindent $(\ast)$ \ \ \  \ \ \ \ for any  $a \in K^{\times}$, the set  $V(a) = \{v \in V \mid v(a) \neq 0 \}$  is finite.
\vskip2mm
\noindent We recall that the ring of adeles $\mathbb{A}(K,V)$ is defined as the restricted product of the completions $K_v$ for $v \in V$ with respect to the valuation rings $\mathcal{O}_v \subset K_v$. In view of condition $(\ast)$, we have a diagonal embedding $K \hookrightarrow \mathbb{A}(K,V)$ that gives $\mathbb{A}(K,V)$ a natural structure of a $K$-algebra. Furthermore, we let
$$
\mathbb{A}^{\infty}(K,V) = \prod_{v \in V} \mathcal{O}_v
$$
be the subring of integral adeles.

Next, let $L/K$ be a finite separable field extension and denote by $V^{L}$ the set of all extensions to $L$ of the discrete valuations in $V$. Then we have an isomorphism of topological rings $$\mathbb{A}(K,V) \otimes_K L \simeq \mathbb{A}(L, V^L).$$ In particular, if $L/K$ is a finite Galois extension, the adele ring $\mathbb{A}(L, V^L)$ has a natural action of the Galois group $\Gal(L/K)$ such that
$$
\mathbb{A}(L, V^L)^{\Gal(L/K)} = \mathbb{A}(K,V).
$$

Suppose now that $G$ is a linear algebraic $K$-group with a fixed $K$-defined matrix realization $G \subset {\rm GL}_n$. For each $v \in V$, we set $G(\mathcal{O}_v) = G(K_v) \cap {\rm GL}_n(\mathcal{O}_v)$, and then define the corresponding adelic group as
$$
G(\mathbb{A}(K,V)) := \left\{ (g_v) \in \prod_{v \in V} G(K_v) \mid g_v \in G(\mathcal{O}_v) \ \text{for almost all} \ v \in V \right\}
$$
(in other words, $G(\mathbb{A}(K,V))$ is the restricted product of the groups $G(K_v)$ with respect to the subgroups $G(\mathcal{O}_v)$)\footnotemark. \footnotetext{We note that condition $(\ast)$ implies that this adelic group does not actually depend on the choice of matrix realization $G \hookrightarrow {\rm GL}_n$; more precisely, a $K$-defined isomorphism between two linear $K$-groups induces an isomorphism between the corresponding adelic groups.}
The product
$$
G(\mathbb{A}^{\infty}(K,V)) := \prod_{v \in V} G(\mathcal{O}_v)
$$
is usually referred to as the \emph{subgroup of integral adeles}. Again, due to the fact that $V$ satisfies $(\ast)$, we have the diagonal embedding $G(K) \hookrightarrow G(\mathbb{A}(K,V))$, whose image is called the group of \emph{principal adeles} and is routinely identified with $G(K)$. For a finite separable extension $L/K$, the groups $G(\mathbb{A}(L, V^L))$ and $G(\mathbb{A}^{\infty}(L, V^L))$ are defined analogously. If $L/K$ is a Galois extension, then the standard action of $\Gal(L/K)$ on $G(L)$ naturally extends to an action on $G(\mathbb{A}(L, V^L))$.

We now specialize to the following situation. Suppose $K = k(X)$ is the function field of a smooth geometrically integral affine variety $X$ defined over a field $k$ of characteristic 0 and let $V$ be the set of geometric places of $K$. Let $D$ be a linear algebraic $k$-group whose connected component $D^0 = T$ is an algebraic $k$-torus. Take $\ell/k$ to be a finite Galois extension that splits $T$, set $L = \ell \cdot K$, and let $V^L$ be the set of all extensions of the valuations in $V$ to $L$. We note that $V^L$ is also a geometric set of places since its elements correspond to the prime divisors on the normalization of $X$ in $L$.

\begin{lemma}\label{L:condition-T}
Keeping the preceding notations, we have the following:

\begin{itemize}
\item[{\rm (i)}] $D(L_v) = D(\mathcal{O}_{L_v}) \cdot T(L_v)$ for almost all $v \in V^L$;

\vskip2mm

\item[{\rm (ii)}] if $k$ is a finitely generated field, then there exists a finite subset $S^L \subset V^L$ such that
$$
D(\mathbb{A}(L, V^L \setminus S^L)) = D(\mathbb{A}^{\infty}(L, V^L \setminus S^L)) \cdot D(L).
$$
\end{itemize}

\end{lemma}

\begin{proof} \ \

\noindent (i) This follows from \cite[Lemma 3.6]{RR-Tori1}.

\vskip2mm

\noindent (ii) The argument at the start of \cite[\S3]{RR-Tori2} (which ultimately boils down to the finite generation of the Picard group ${\rm Pic}(L, V^L)$) shows that there exists a finite subset $S^L \subset V^L$ such that
$$
T(\mathbb{A}(L, V^L \setminus S^L)) = T(\mathbb{A}^{\infty}(L, V^L \setminus S^L)) \cdot T(L).
$$
We may assume that $S^L$ is sufficiently large so that (i) holds for all $v \in V^L \setminus S^L$. Then the natural map
$$
T(\mathbb{A}^{\infty}(L, V^L \setminus S^L))\backslash T(\mathbb{A}(L, V^L \setminus S^L))/ T(L) \to D(\mathbb{A}^{\infty}(L, V^L \setminus S^L))\backslash D(\mathbb{A}(L, V^L \setminus S^L))/ D(L)
$$
is surjective; since the term on the left reduces to a single element, so does the term on the right, which yields our claim.
\end{proof}

Next, the diagonal embedding $D(L) \hookrightarrow D(\mathbb{A}(L, V^L))$ induces a map
$$
\lambda_{D, V, L/K} \colon H^1(L/K, D) \to H^1(L/K, D(\mathbb{A}(L, V^L))).
$$
Furthermore, for each $v \in V$, let us pick one extension $w \in V^L$, and consider the global-to-local map
$$
\theta_{D,V,L/K} \colon H^1(L/K, D) \to \prod_{v \in V} H^1(L_w/K_v, D).
$$
We define the \emph{relative Tate-Shafarevich set} (with respect to the extension $L/K$ and the set $V$) as
$$
\Sha(L/K, D, V) = \ker \theta_{D,V, L/K}.
$$
One can make similar definitions when $V$ is replaced by a subset $V' \subset V$ with finite complement (and $V^L$ by the corresponding set of extensions $(V')^L$).

\begin{lemma}\label{L:Relative-Sha-property}
With these notations, we have the following.
\begin{itemize}

\item[{\rm (i)}] $\Sha(L/K, D, V)$ coincides with the kernel of $\lambda_{D, V, L/K}$.

\vskip2mm

\item[{\rm (ii)}] Set $E(D, V, L) = D(L) \cap D(\mathbb{A}^{\infty}(L, V^L))$. If $D(\mathbb{A}(L, V^L)) = D(\mathbb{A}^{\infty}(L, V^L)) \cdot D(L)$, then $\Sha(L/K, D, V)$ is contained in the image of the map
    $$
    \nu \colon H^1(L/K, E(D, V, L)) \to H^1(L/K, D)
    $$
    induced by the inclusion $E(D, V, L) \hookrightarrow D(L).$

\end{itemize}

\end{lemma}
\begin{proof} \ \

\noindent (i) This is proved along the same lines as \cite[Proposition 3.1]{RR-Tori1}. Namely, as noted in that argument, it suffices to shows that for almost all $v \in V$ and $w \vert v$, the natural map
$$
\mu^1 \colon H^1(L_w/K_v, D(\mathcal{O}_{L_w})) \to H^1(L_w/K_v, D(L_w))
$$
has trivial kernel.
For this, using the fact that the extension $L_w/K_v$ is unramified, one first observes that
$$
T(L_w) \simeq \Delta \times T(\mathcal{O}_{L_w}), \ \ \ \text{where} \ \Delta = X_*(T) \ \text{is the group of co-characters}.
$$
Consequently, by Lemma \ref{L:condition-T}(i), we have
$$
D(L_w) = \Delta \cdot D(\mathcal{O}_{L_w})
$$
for almost all $v \in V$ and $w \vert v$, and, moreover, $\Delta \cap D(\mathcal{O}_{L_w})$ since $\Delta$ does not contain any bounded subgroups. Using these facts, one then checks that $\mu^1$ has trivial kernel by a direct calculation.

\vskip2mm

\noindent (ii) Suppose $x \in H^1(L/K, D)$ lies in  $\ker \lambda_{D, V, L/K}$, and represent $x$ by a cocycle $\{x_{\sigma}\}_{\sigma \in \Gal(L/K)}.$ Then there exists $a \in D(\mathbb{A}(L, V^L))$ such that
$$
x_{\sigma} = a^{-1} \sigma(a) \ \ \ \text{for all} \ \sigma \in \Gal(L/K).
$$
By our assumption, we can write
$$
a = b \cdot c, \ \ \ \text{with} \ b \in D(\mathbb{A}^{\infty}(L, V^L)) \ \text{and} \ c \in D(L).
$$
Then
$$
y_{\sigma} = c x_{\sigma} \sigma(c)^{-1}
$$
defines a cocycle in the same cohomology class as $x$ that clearly takes values in $E(D, V^L, L)$, proving our claim.

\end{proof}

Now, if $k$ is a finitely generated field, then by Lemma \ref{L:condition-T}(ii), we can find a subset $V' \subset V$ with finite complement such that for the corresponding set of extensions $(V')^L$, we have
$$
D(\mathbb{A}(L, (V')^L)) = D(\mathbb{A}^{\infty}(L, (V')^L)) \cdot D(L).
$$
Combining this with Lemma \ref{L:Relative-Sha-property}(ii), we obtain

\begin{cor}\label{C:Relative-Sha-unit-group}
Suppose $k$ is a finitely generated field. Then $\Sha(L/K, D, V')$, and hence also $\Sha(L/K, D, V)$, lies in the image of
$$
\nu' \colon H^1(L/K, E(D, V', L)) \to H^1(L/K, D),
$$
where $E(D, V', L) = D(L) \cap D(\mathbb{A}^{\infty}(L, (V')^L))$.

\end{cor}

\section{Finiteness of the relative Tate-Shafarevich group}\label{S:Finiteness-Rel-TS}

We now return to our consideration of groups of multiplicative type. More precisely, we will work with the following setup. Let $K = k(X)$ be the function field of a smooth geometrically integral affine variety $X$ of dimension $\geq 2$ over a finitely generated field $k$ of characteristic 0, and let $V$ be the set of geometric places of $K$. Suppose $D$ is a group of multiplicative type defined over $k$. Let $\ell/k$ be a finite Galois extension that splits $T = D^{\circ}.$
Set $L = \ell \cdot K$ and let $V^L$ be the set of all extensions of the valuations in $V$ to $L$. Our goal in this section is to prove the following.

\begin{prop}\label{P:Finiteness-of-Rel-TS}
The map
$$
\lambda_{D, V, L/K} \colon H^1(L/K, D) \to H^1(L/K, D(\mathbb{A}(L, V^L)))
$$
is proper in the following cases:

\vskip2mm
\begin{enumerate}
        \item $k$ is a finitely generated field and $X(k) \neq \emptyset$;
        \item $k$ is a number field.
    \end{enumerate}

\end{prop}

\vskip2mm

It suffices to show that $\ker \lambda_{D, V, L/K}$ (which, according to Lemma \ref{L:Relative-Sha-property}(i), coincides with the relative Tate-Shafarevich group $\Sha(L/K, D, V)$) is finite.

Now, by Corollary \ref{C:Relative-Sha-unit-group}, we can find a subset $V' \subset V$ with finite complement such that $\ker \lambda_{D, V, L/K}$ is contained in the image of
$$
\nu' \colon H^1(L/K, E(D, V', L)) \to H^1(L/K, D),
$$
where $E(D, V', L) = D(L) \cap D(\mathbb{A}^{\infty}(L, (V')^L))$.

\begin{lemma}\label{L:FG}
With notations as above, the quotient $E(D, V', L)/D(\ell)$ is a finitely generated abelian group.
\end{lemma}
\begin{proof}
According to \cite[Lemma 6]{RR-Tori2}, the analogously defined group $E(T,V',L)/T(\ell)$ for $T = D^{\circ}$ is finitely generated. Since $E(D, V', L)/D(\ell)$ contains $E(T,V',L)/T(\ell)$ as a subgroup of finite index, our claim follows.

\end{proof}

We thus have a short exact sequence of modules over $\Gal(\ell/k) = \Gal(L/K)$
$$
1 \to D(\ell) \to E(D, V', L) \to \Gamma \to 1,
$$
where $\Gamma$ is finitely generated as an abelian group. This leads to the following exact sequence of cohomology groups:
$$
H^1(\ell/k, D) \stackrel{\delta}{\longrightarrow} H^1(L/K, E(D, V', L)) \to H^1(L/K, \Gamma).
$$
Since the group $H^1(L/K , \Gamma)$ is finite (see, for example, \cite[Ch. II, Corollary 1.32]{MilneCFT}), it follows that the intersection $\ker(\lambda_{D, V, L/K}) \cap ({\rm Im}(\nu' \circ \delta))$ has finite index in $\ker(\lambda_{D, V, L/K})$.

Let us now define
$$
\Sha_0(D, V) = \ker \left( H^1(\ell/k, D) \to \prod_{v \in V} H^1(L_w/ K_v, D) \right),
$$
where, as in the discussion in \S\ref{S:adeles}, for each $v \in V$, we fix a \emph{single} extension $w \vert v$ in $V^L.$ From our definitions, it is clear that
$$
\ker(\lambda_{D, V, L/K}) \cap ({\rm Im}(\nu' \circ \delta)) = (\nu' \circ \delta)(\Sha_0(D, V)).
$$
Therefore, Proposition \ref{P:Finiteness-of-Rel-TS} will follow from

\begin{prop}\label{P:Finiteness-of-Rel-TS-constant}
The group $\Sha_0(D,V)$ is finite in the following cases:

\vskip2mm
\begin{enumerate}
        \item $k$ is a finitely generated field and $X(k) \neq \emptyset$;
        \item $k$ is a number field.
    \end{enumerate}

\end{prop}

\vskip2mm

\noindent {\bf Case 1: $X(k) \neq 0$.} For the argument, we will need the following statement.

\begin{lemma}\label{L:injective-map-rational-point}
Let $X$ be a smooth geometrically integral variety of dimension $\geq 2$ over a field $k$ with function field $K = k(X)$. Suppose $Y \subset X$ is a smooth geometrically integral $k$-defined subvariety of codimension 1, and let $v$ be the discrete valuation of $K$ associated with $Y$. Let $D$ be a $k$-defined group of multiplicative type and $\ell/k$ be a finite Galois extension.
If $Y(k) \neq \emptyset$, then the map
$$
\mu_v \colon H^1(\ell/k, D) \to H^1(L_w/K_v, D), \ \ \ \text{where} \ L = \ell \cdot K \ \text{and} \ w \vert v,
$$
is injective.
\end{lemma}
\begin{proof}
The proof closely resembles the argument given at the end of the proof of Proposition \ref{P:propfin}. We sketch the main points for the sake of completeness. We begin with the following diagram
\[\begin{tikzcd}
	{H^1(\ell/k, D)} & {H^1(L_w/K_v, D(\mathcal{O}_{L_w}))} & {H^1(L_w/K_v, D(L_w))} \\
	& {H^1(L^{(w)}/K^{(v)}, D)}
	\arrow["{\mu_v'}", from=1-1, to=1-2]
	\arrow["{\nu_v}"', dashed, from=1-1, to=2-2]
	\arrow["{\mu_v''}", from=1-2, to=1-3]
	\arrow["{\rho_v}", from=1-2, to=2-2]
\end{tikzcd}\]
where $\mu_v = \mu_v'' \circ \mu_v'$ is the natural factorization induced by the inclusions $D(\ell) \hookrightarrow D(\mathcal{O}_{L_w}) \hookrightarrow D(L_w)$,  $L^{(w)}$ and $K^{(v)}$ denote the corresponding residue fields, and $\rho_v$ is induced by reduction modulo the maximal ideal of the valuation ring $\mathcal{O}_{{L}_w}$.

The map $\mu_v''$ is injective by \cite[Theorem 4.1]{CT-S}, so it follows that $\ker \mu_v$ is contained in the kernel of $\nu_v = \rho_v \circ \mu_v'.$ Since $K^{(v)} = k(Y)$ and $L^{(w)} = \ell(Y_{\ell})$, where $Y_{\ell} = Y \times_k \ell$, we see that to conclude the proof, it suffices to establish the injectivity of the natural map
$$
\nu_Y \colon H^1(\ell/k, D) \to H^1(\ell(Y_{\ell})/k(Y), D).
$$
Fix $y \in Y(k) \subset Y_{\ell}(\ell)$ and
consider the local rings $\mathcal{O}_{Y, y}$ and $\mathcal{O}_{Y_{{\ell}}, y} = {\ell} \cdot \mathcal{O}_{Y, y}$ with maximal ideals $\mathfrak{m}$ and $\mathfrak{m}_{{\ell}}.$ The residue fields $\mathcal{O}_{Y,y}/\mathfrak{m}$ and $\mathcal{O}_{Y_{{\ell}}, y}/\mathfrak{m}_{{\ell}}$ coincide with $k$ and ${\ell}$. We have a similar diagram to the one above
\[\begin{tikzcd}
	{H^1(\ell/k, D)} & {H^1(\ell(Y_{\ell})/k(Y), D(\mathcal{O}_{Y_{\ell}, y}))} & {H^1(\ell(Y_{\ell})/k(Y), D(\ell(Y_{\ell})))} \\
	& {H^1(\ell/k, D)}
	\arrow["{\nu_{Y}'}", from=1-1, to=1-2]
	\arrow["{{\rm id}}"', dashed, from=1-1, to=2-2]
	\arrow["{\nu_Y''}", from=1-2, to=1-3]
	\arrow["{\rho_{Y_{\ell},y}}", from=1-2, to=2-2]
\end{tikzcd}\]
where again the top row is a factorization of $\nu_Y$ and $\rho_{Y_{\ell}, y}$ is induced by reduction modulo $\mathfrak{m}_{\ell}$. Since $Y$ is smooth, the map $\nu_Y''$ is injective by \cite[Theorem 4.1]{CT-S}. On the other hand, the composition $\rho_{Y_{\ell}, y} \circ \nu_Y'$ is the identity map, so $\nu_Y'$ is injective, and hence $\nu_Y$ is injective, as needed.

\end{proof}

Suppose now that $x \in X(k)$. Using Proposition \ref{P:Bertini-type}, we find a geometrically integral smooth $k$-defined subvariety $Y \subset X$ of codimension 1 such that $x \in Y(k)$. Taking $v$ to be the discrete valuation of $K = k(X)$ associated with $Y$, we conclude from Lemma \ref{L:injective-map-rational-point} that the map
$$
\mu_v \colon H^1(\ell/k, D) \to H^1(L_w/K_v, D)
$$
is injective, and hence $\Sha_0(D,V)$ is trivial.

\vskip2mm

\noindent {\bf Case 2: $k$ is a number field.} By Proposition \ref{P:Bertini-type}, we can find a geometrically integral smooth $k$-defined subvariety $Y \subset X$ of codimension 1. Let $v$ be the discrete valuation of $K = k(X)$ associated with $Y$. As a consequence of the Lang-Weil estimates (cf. \cite{LW}) and Hensel's lemma, there exists a subset $U \subset V^k_f$ with finite complement $V^k_f \setminus U$ (where $V^k_f$ is the set of finite places of $k$) such that $Y(k_u) \neq \emptyset$ for all $u \in U.$

For each $u \in U$, let $\bar{u} \in V^{\ell}_f$ be an extension to $\ell$, and define
$$
\Sha_{\ell/k}(D, U) := \ker \left( H^1(\ell/k, D) \to \prod_{u \in U} H^1(\ell_{\bar{u}}/k_u, D) \right).
$$
We will show that $\Sha_0(D,V) \subset \Sha_{\ell/k}(D, U)$. This will complete the argument since, on the one hand, $\Sha_{\ell/k}(D,U)$ is a subgroup of
$$
\Sha_k(D, U) := \ker \left( H^1(k,D) \to \prod_{u \in U} H^1(k_u, D) \right)
$$
by the injectivity of the inflation map in degree 1, and, on the other hand, $\Sha_k(D,U)$ is finite by a theorem of Borel and Serre (see \cite[Ch. III, \S4.6, Theorem 7]{Serre-GC}).

To establish the inclusion $\Sha_0(D,V) \subset \Sha_{\ell/k}(D, U)$, we take $u \in U$, set
$$
X_u = X \times_k k_u \ \ \ \text{and} \ \ \ Y_u = Y \times_k k_u,
$$
and let $v_u$ be the discrete valuation of $K^u := k_u(X_u)$ associated with $Y_u$. We then have the commutative diagram
\[\begin{tikzcd}
	{H^1(\ell/k, D)} & {H^1(L_w/K_v, D)} \\
	{H^1(\ell_{\bar{u}}/k_v, D)} & {H^1((L^u)_{w_{u}}/(K^u)_{v_u}, D)}
	\arrow["{\mu_v}", from=1-1, to=1-2]
	\arrow["{\theta_u}"', from=1-1, to=2-1]
	\arrow["{\Theta_u}", from=1-2, to=2-2]
	\arrow["{\mu_{v_u}}", from=2-1, to=2-2]
\end{tikzcd}\]
where $L^u = \ell \cdot K^u$ and $w_u \vert v_u.$ Suppose $\xi \in \Sha_0(D,V).$ Then
$$
\mu_{v_u}(\theta_u(\xi)) = \Theta_u(\mu_v(\xi)) = 1.
$$
Since $Y(k_u) \neq \emptyset$, the map $\mu_{v_u}$ is injective by Lemma \ref{L:injective-map-rational-point}, and hence $\xi \in \ker(\theta_u)$. Since $u \in U$ was arbitrary, this yields the inclusion $\Sha_0(D,V) \subset \Sha_{\ell/k}(D, U)$ and concludes the proof.

\section{Proof of Theorem \ref{T:mainthm}}\label{S:proof-of-main-theorem}

In this section, we will apply the results of the previous sections to prove Theorem \ref{T:mainthm}. Recall that $K = k(X)$ is the function field of a smooth geometrically integral variety $X$ of dimension $\geq 2$ over a field $k$ of characteristic 0, $V$ is the set of geometric places of $K$, and $D$ is a $k$-defined group of multiplicative type. We would like to show that $\Sha(D,V)$ is finite in the two cases listed in Theorem \ref{T:mainthm}.

Let $T = D^{\circ}$ (a $k$-torus) and $\Omega = D/T$ (a finite $k$-group). As before, let $\ell/k$ be a finite Galois extension such that $T$ splits over $\ell$.
Set $L = \ell \cdot K$ and let $V^L$ be the set of all extensions of the valuations in $V$ to $L$.

We first note that the inflation-restriction sequence yields an exact sequence
$$
0 \to \Sha(L/K, D,V) \to \Sha(K, D) \to \Sha(L, D),
$$
where
$$
\Sha(L,D) = \ker \left(H^1(L,D) \to \prod_{w \in V^L} H^1(L_w, D) \right).
$$
According to Proposition \ref{P:Finiteness-of-Rel-TS}, the group $\Sha(L/K, D, V)$ is finite in cases (1) and (2). To show that $\Sha(L,D)$ is finite, we consider the exact sequence
$$
1 \to T \to D \to \Omega \to 1.
$$
Since $H^1(L, T) = 0$ by Hilbert's Theorem 90, and we have the commutative diagram
\[\begin{tikzcd}
	{H^1(L,D)} & {H^1(L, \Omega)} \\
	{\displaystyle{\prod_{w \in V^L}}H^1(L_w, D)} & {\displaystyle{\prod_{w \in V^L}}H^1(L_w, \Omega)}
	\arrow[from=1-1, to=1-2]
	\arrow["{\theta_{D,V^L}}"', from=1-1, to=2-1]
	\arrow["{\theta_{\Omega, V^L}}", from=1-2, to=2-2]
	\arrow[from=2-1, to=2-2]
\end{tikzcd}\]
arising from the above sequence, we obtain an inclusion
$$
\Sha(L,D) \hookrightarrow \Sha(L, \Omega).
$$
By Proposition \ref{P:propfin}, the group $\Sha(L, \Omega)$ is finite, yielding the finiteness of $\Sha(L,D)$, and hence also the finiteness of $\Sha(K, D)$, as needed.

\vskip4mm

\noindent {\bf Acknowledgements.} The authors were partially supported by NSF grant DMS-2154408 during the preparation of this paper.

\end{document}